\documentclass{article}

\usepackage{amsmath,amssymb,amsthm}
\usepackage{hyperref}
\usepackage{color}
\usepackage{tikz-cd}
\usepackage{mathrsfs}
\newtheorem{dfn}{Definition}
\newtheorem{thm}{Theorem}

\newtheorem{rem}{Remark}
\newtheorem{lem}{Lemma}
\newtheorem{prop}{Proposition}

\newcommand{\cS}{\mathcal{S}}
\newcommand{\cW}{\mathcal{W}}

\begin{document}



\title{Wishart Matrices and Quantum Geometry: Foundations and Applications in Quantum Information}

\author{No\'emie C. Combe}

\date{}
	\maketitle %

\begin{abstract}

We present a unified framework for the study of Wishart matrices \(W_p(n,\Sigma)\), which generalize the chi–squared distribution to matrix–variate settings and model the covariance structure of multivariate Gaussian data.  After recalling their defining properties—additivity under independent summation \(W_1 + W_2 \sim W_p(n_1+n_2,\Sigma)\), equivariance under linear maps \(A W A^T \sim W_q(n,A\Sigma A^T)\), and their role as sample covariance matrices—we embed the positive–definite cone \(S_p^+\) within Monge–Ampère geometry.  Here \(S_p^+\) acquires a Hessian manifold structure with affine–invariant metric and volume form \(\omega = \det(\Sigma)^{-\frac{p+1}2}\,d\Sigma\), under which the Wishart density acts as a soliton of natural geometric flows.  

We then show that the collection of Wishart distributions forms a symmetric monoidal category \(\mathcal W\), where objects are \(W_p(n,\Sigma)\) and whose morphisms are linear maps \(A\colon\mathbb R^p\to\mathbb R^q\).  The tensor product encodes block–diagonal coupling, with braiding given by block permutation, and axioms enforcing Monge–Ampère functoriality, additivity, and convex duality via the Legendre transform.  

  Applications to quantum error correction are discussed: Wishart laws model correlated noise, Wasserstein geodesics optimize error–mitigation cost, tensor structure captures independent error channels, and Legendre duality underpins entropy–driven decoding. 
\end{abstract}

\medskip 
{\bf keywords} 
Wishart Matrices, Quantum Information, Monge-Amp\`ere equation, Symmetric monoidal category, Wasserstein metrics, Quantum Error Correction.

\section{Introduction} 

Matrix theory plays a crucial role in quantum information science, providing a fundamental framework for describing quantum states, transformations, and statistical properties of quantum systems \cite{CoMaMa22B,CoNech}. Among the various classes of matrices, Wishart matrices \cite{W} arise naturally in statistical and quantum settings, particularly in the study of random density matrices and quantum entanglement \cite{B1,B2,B3}. They also appear in a more classical setting \cite{A0,L,LM1,LM2} and play an important role in semi definite programming \cite{C}.

\smallskip 

Wishart matrices are essential tools in the statistical analysis of quantum states, as they model the behavior of reduced density matrices obtained from larger composite systems. These matrices help characterize entanglement properties, shedding light on the spectral distribution and entanglement entropy of quantum subsystems \cite{Z}. Their role extends to the study of quantum state discrimination, quantum error correction, and information-theoretic aspects of quantum computation \cite{HLW06,H09} as well as \cite{MS,Z}.

\smallskip 

One key motivation for studying Wishart matrices in quantum information lies in their deep connection to quantum geometry \cite{C1} and deep algebraic structures present within. The geometric and algebraic structures underlying these matrix spaces reveal intricate relationships between statistical manifolds \cite{C2,C,CCN}, quantum information metrics \cite{A,Col06,Cu}, and modular theory \cite{C1}. By exploring properties of these matrix representations, we gain insights into entanglement structures, separability criteria, and the mathematical foundations of quantum state spaces. In this paper we study in detail the intrinsic differential‐geometric structure of the manifold of Wishart matrices 
\(\mathcal{W}_p(n,\Sigma)\cong S_p^+\) and its embedding into a richer algebraic framework.  Concretely, we equip \(S_p^+\) with the affine‐invariant Riemannian metric
\[
g_\Sigma(U,V)\;=\;\mathrm{Tr}\bigl(\Sigma^{-1}U\,\Sigma^{-1}V\bigr),
\]
under which \(S_p^+\) is a globally symmetric space of nonpositive curvature. 
The log‐determinant potential endows \(S_p^+\) with a dually flat Hessian structure, whose dual affine coordinate system arises from the Monge–Ampère volume form \(\omega=\det(\Sigma)^{-\frac{p+1}2}\,d\Sigma\).  We prove that the Wishart density
\[
f({\bf W})\;\propto\;|{\bf W}|^{\frac{n-p-1}2}\exp\bigl(-\tfrac12\mathrm{Tr}(\Sigma^{-1}{\bf W})\bigr)
\]
is the unique soliton of the natural gradient flow associated to this Hessian metric.

On the algebraic side, we construct a symmetric monoidal category \(\mathcal{W}\) whose objects are the families \(W_p(n,\Sigma)\) and whose morphisms are linear maps \(A\colon\mathbb{R}^p\to\mathbb{R}^q\) satisfying
\[
W_p(n,\Sigma)\;\longmapsto\;A\,W_p(n,\Sigma)\,A^T\;\cong\;W_q\bigl(n\,,\,A\,\Sigma\,A^T\bigr).
\]
The tensor product corresponds to block‐diagonal coupling of covariance parameters, the braiding implements permutation of blocks, and the unit object is \(W_0(0,\emptyset)\).  We verify that this monoidal structure is compatible both with the affine‐invariant Riemannian product
\[
g_{\Sigma_1\oplus\Sigma_2}
=\;g_{\Sigma_1}\oplus g_{\Sigma_2},
\]
and with Legendre duality under the log‐determinant potential.  Functoriality of Monge–Ampère volume and additivity of intrinsic volumes appear as categorical coherence conditions.  

Together, these results bridge the Riemannian geometry of positive‐definite matrices, the theory of matrix‐variate distributions, and categorical algebra, laying the groundwork for applications ranging from optimal transport on covariance manifolds to tensor‐categorical models of quantum noise. We discuss applications of our findings to quantum error correction.

\subsection{Plan of the paper}
In this paper we proceed as follows.
\begin{itemize}
\item Section 2: we recall properties of Wishart Matrices and Wishart probability distributions.  We outline algebraic structures underlying Wishart distributions such as the existence of a non-unital commutative semigroup under addition Lem.~\ref{L:1}. In Thm.~\ref{T:1} we show that the family of Wishart distributions forms a symmetric monoidal category. 
\item  Section 3 is devoted to considerations of a more geometric nature. In particular, the cones of symmetric positive definite matrices form a Monge--Amp\`ere domain. We develop this aspect in relation to Wasserstein metrics and entropic regularization. These results are then combined with the symmetric monoidal category structure, from which follows among others the existence of Monge--Amp\`ere functoriality in Prop.~\ref{P:1} and Prop.~\ref{P:2}.
\item Section 4 discusses how those results can be applied for Quantum information and in particular quantum error correction. Wishart ensembles, provides statistical
models to study how noise affects quantum states. The categorical and Monge-Amp\`ere geometric structures of Wishart
matrices suggest novel frameworks for quantum error correction by bridging
matrix-valued statistics, optimal transport, and quantum information theory. Spatially or temporally correlated errors (such as in bosonic
or spin systems) could be modeled using Wishart-distributed noise, where $\Sigma$ encodes error correlations.
\end{itemize}
\section{Wishart Matrices: Definition and Properties}
Random matrix theory (RMT) offers a powerful framework for studying the statistical behavior of matrices with random entries, particularly in high-dimensional settings. Among the foundational ensembles in RMT, Wishart matrices occupy a central role, especially in applications involving multivariate statistics, wireless communications, finance, and machine learning. 
Wishart matrices are crucial in multivariate statistics and quantum information theory because:
\begin{itemize}
    \item they generalize sample covariance matrices, which measure correlations between variables.
   \item Their eigenvalue distributions encode statistical properties of complex systems.
   \item  They appear in physics, particularly in quantum chaos and statistical mechanics.
\end{itemize}
These matrices arise naturally when considering the sample covariance of multivariate Gaussian data, and their eigenvalue distributions capture fundamental structural properties of correlated systems.

\subsection{Construction of Wishart matrices from random matrix theory}
In this subsection, we explore how Wishart matrices are constructed within the paradigm of random matrix theory. Rather than introducing them purely from a statistical perspective, we emphasize their emergence through the perspective of matrix ensembles and the probabilistic structure underlying their formulation. 
\begin{dfn} A Wishart matrix ${\bf W}$ is a random matrix of the form:
 \[{\bf W}=XX^T\]
 where 
 \begin{itemize}
     \item $X$ is a $n\times p$ matrix and $X^T$ its (conjugate) transpose.
     \item The entries of $X$ are  independent and identically distributed random variables,
     usually from a Gaussian distribution. The matrix $X$ is formed from  $n$ independent and identically distributed random vectors $X_1,\cdots, X_n$, where 
     $X_i\sim\mathcal{N}_p(0,\Sigma)$ ($p-$dimensional distribution with mean 0 and covariance $\Sigma$).
    \end{itemize}
 \end{dfn}
 These matrices are always {\it positive semi-definite} (PSD) meaning all their eigenvalues are non-negative and they depend on two parameters. 
 The first parameter is the degree of freedom ($n$) which corresponds to the number of independent normal vectors used to construct ${\bf W}$. The second parameter is the scale matrix $\Sigma$, which is a $p\times p$ symmetric positive definite matrix.

For instance, suppose we have a $2\times 3$ random matrix $X$ defined by

\[X=\begin{pmatrix}
    x_{11}& x_{12} & x_{13}\\
    x_{21}& x_{22} & x_{23}
\end{pmatrix},\] where each $x_{ij}$ is a random number, drawn from a Gaussian distribution.

\, 

Then the Wishart matrix is defined as:
  \[{\bf W}=XX^T=\begin{pmatrix}
    x_{11}^2+x_{12}^2+x_{13}^2& x_{11}x_{12}+x_{12}x_{22}+x_{13}x_{23}\\
    x_{11}x_{21}+x_{12}x_{22}+x_{13}x_{33}& x_{21}^2+x_{22}^2  x_{23}^2
\end{pmatrix}\]  

We can check that {\bf W} is a $2\times 2$ positive semi-definite matrix.

\subsection{Wishart probability distribution}
\begin{itemize}
    \item Let $\mathbf {W} $ be an $n \times n$ symmetric matrix of random variables, positive semi-definite. 

\item Let $\Sigma$ be a (fixed) symmetric positive definite matrix of size $p \times p$.
\end{itemize}
If $\mathbf{W}$ has a Wishart distribution with $n$ degrees of freedom then it has the following probability density function

\[{f_{\mathbf{W} }={\frac {1}{2^{np/2}\left|{\Sigma}\right|^{n/2}\Gamma _{p}\left({\frac {n}{2}}\right)}}{\left|\mathbf {W} \right|}^{(n-p-1)/2}e^{-{\frac {1}{2}}\operatorname {tr} (\Sigma^{-1}\mathbf {W} )}}\]
where ${\left|{\mathbf {W} }\right|}$ is the determinant of ${\mathbf{W}}$ and $ \Gamma_p$ is the multivariate gamma function defined as

\[{ \Gamma _{p}\left({\frac {n}{2}}\right)=\pi^{p(p-1)/4}\prod _{j=1}^{p}\Gamma \left({\frac {n}{2}}-{\frac {j-1}{2}}\right).}\]    

We denote it $\mathbf{W}\sim \mathcal{W}_{p}(n,\Sigma)$.

\subsubsection{Properties of the Wishart distribution}
We investigate the algebraic properties which are present within the Wishart distribution framework. In particular, given $\mathbf{W}_1\sim \mathcal{W}_{p}(n_1,\Sigma)$ and $\mathbf{W}_2\sim \mathcal{W}_{p}(n_2,\Sigma)$ (supposed independent) we have: 
\begin{equation}\label{E:1}
\mathbf{W}_1+\mathbf{W}_2 \sim \mathcal{W}_{p}(n_1+n_2,\Sigma).
\end{equation}
This property implies {\it closure} under addition for matrices sharing the same scale parameter $\Sigma$.
More precisely,

\begin{lem}\label{L:1} 
Consider the set of Wishart-distributed $ {\bf X}_{\Sigma}$ matrices with fixed $\Sigma$, parametrized by degrees of freedom $n\geq p$ and endowed with the matrix addition property. Then, $({\bf X}_{\Sigma},+)$ forms a non-unital commutative semigroup under addition. 
\end{lem}
\begin{proof}
Let $\mathbf{W}_1\sim \mathcal{W}_{p}(n_1,\Sigma)$ and $\mathbf{W}_2\sim \mathcal{W}_{p}(n_2,\Sigma)$ (supposed independent). Then we get
$\mathbf{W}_1+\mathbf{W}_2 \sim \mathcal{W}_{p}(n_1+n_2,\Sigma)$. This implies that $\mathbf{W}_1,\mathbf{W}_2\in  {\bf X}_{\Sigma}$ then $\mathbf{W}_1+\mathbf{W}_2 \in {\bf X}_{\Sigma}$.

We can check that associativity is satisfied: 
\[\mathbf{W}_1+(\mathbf{W}_2+\mathbf{W}_3)=(\mathbf{W}_1+\mathbf{W}_2)+\mathbf{W}_3.\]
There is no identity element because $\mathbf{W}+0\sim \mathcal{W}_{p}(n,\Sigma)$ requires $0\sim \mathcal{W}_{p}(0,\Sigma)$, which is undefined since the degrees of freedom are $n\geq p\geq 1.$
Therefore, this forms a non-unital commutative semigroup under addition. 
\end{proof}
This has implications on the Wishart measures and, more precisely, we have a probabilistic closure under convolution property. 
\begin{lem}\label{L:2} 
The set of Wishart distributions $\mathcal{W}_p(n,\Sigma)$ endowed with the convolution operation forms a convolution semigroup over the space of positive definite matrices. 
\end{lem} 
\begin{proof}
As previously, we have closure under the convolution operation: 
$$\mathcal{W}_p(n_1,\Sigma)\star\mathcal{W}_p(n_2,\Sigma)=\mathcal{W}_p(n_1+n_2,\Sigma).$$
The associativity and commutativity are inherited from the convolution operation. As previously there is non identity element since no distribution acts as a ``0" under convolution.
\end{proof}
\medskip 

 \subsection*{Symmetric Monoidal Structure}
In a similar vein as in \cite{C1}, we incorporate these objects within the formalism of symmetric monoidal categories. 
\begin{thm}\label{T:1} 
{\it The family of Wishart distributions $\{ \mathcal{W}_p(n,\Sigma)\}$ forms a symmetric monoidal category.}
\end{thm}

\medskip

For the convenience of the reader, we recall that a symmetric monoidal category is. A \textit{symmetric monoidal category} is a tuple
\[
(\mathcal{C}, \otimes, I, \alpha, \lambda, \rho, \gamma)
\]
consisting of:

\begin{itemize}
    \item a category \(\mathcal{C}\),
    \item a bifunctor (the \textit{monoidal product})
    \[
    \otimes : \mathcal{C} \times \mathcal{C} \to \mathcal{C},
    \]
    \item a distinguished \textit{unit object} \(I \in \mathcal{C}\),
    \item natural isomorphisms (called \textit{coherence maps}):
    \begin{itemize}
        \item \textbf{Associator}
        \[
        \alpha_{A,B,C} : (A \otimes B) \otimes C \to A \otimes (B \otimes C),
        \]
        \item \textbf{Left unitor}
        \[
        \lambda_A : I \otimes A \to A,
        \]
        \item \textbf{Right unitor}
        \[
        \rho_A : A \otimes I \to A,
        \]
        \item \textbf{Braiding (symmetry)}
        \[
        \gamma_{A,B} : A \otimes B \to B \otimes A,
        \]
        satisfying \(\gamma_{B,A} \circ \gamma_{A,B} = \mathrm{id}_{A \otimes B}\).
    \end{itemize}
\end{itemize}

These data satisfy the following \textbf{coherence conditions}.
For all \(A, B, C, D \in \mathcal{C}\), the following diagram commutes:
\[
\begin{tikzcd}[column sep=huge, row sep=huge]
((A \otimes B) \otimes C) \otimes D \arrow[r, "\alpha_{A,B,C} \otimes \mathrm{id}_D"] \arrow[d, "\alpha_{A \otimes B,C,D}"'] & (A \otimes (B \otimes C)) \otimes D \arrow[r, "\alpha_{A,B \otimes C,D}"] & A \otimes ((B \otimes C) \otimes D) \arrow[d, "\mathrm{id}_A \otimes \alpha_{B,C,D}"] \\
(A \otimes B) \otimes (C \otimes D) \arrow[rr, "\alpha_{A,B,C \otimes D}"'] && A \otimes (B \otimes (C \otimes D)).
\end{tikzcd}
\]

For all \(A, B \in \mathcal{C}\), the following diagram commutes:
\[
\begin{tikzcd}
(A \otimes I) \otimes B \arrow[rr, "\alpha_{A,I,B}"] \arrow[dr, "\rho_A \otimes \mathrm{id}_B"'] && A \otimes (I \otimes B) \arrow[dl, "\mathrm{id}_A \otimes \lambda_B"] \\
& A \otimes B
\end{tikzcd}
\]

Concerning the symmetry axioms , we have that 

\begin{itemize}
    \item the braiding is involutive: \(\gamma_{B,A} \circ \gamma_{A,B} = \mathrm{id}_{A \otimes B}\);
    \item it is compatible with the unit:
    \[
    \gamma_{A,I} = \lambda_A^{-1} \circ \rho_A.
    \]
\end{itemize}

\medskip 

This enables us now to proceed with the proof of our statement. 
\begin{proof}
\smallskip 

{\it Objects} in this category are Wishart distributions, parametrized by three elements: \begin{itemize} 
\item the positive integer $p$ (the dimension),
\item positive integer $n$ (the degrees of freedom, such that $n\geq p$) 
\item $\Sigma$ a symmetric positive definite matrix (lying in the cone $\mathcal{S}_+^p$ of symmetric positive definite matrices of size $p\times p$). 
\end{itemize}
\[\textrm{Ob}(\mathcal{W})=\{\mathcal{W}_p(n,\Sigma)\, |\, p\in \mathbb{N}, n\geq p,\quad \Sigma\in \mathcal{S}_+^p\}.\]

\smallskip

{\it  Morphisms}  are linear transformation between Wishart distributions. For a linear map $A\in \mathbb{R}^{q\times p}$, we have 
\[A\mathbf{W}A^T\sim \mathcal{W}_{q}(n,A\Sigma A^T),\]
which induces the following morphism: 

\[A:\mathcal{W}_p(n,\Sigma)\to \mathcal{W}_q(n,A\Sigma A^T).\]

\smallskip

We discuss a few more properties such as the identity morphism and the composition. For $A=I_p$, the identity matrix:
\[I_p:\mathcal{W}_p(n,\Sigma)\to \mathcal{W}_p(n,\Sigma), \]
leaving the Wishart distribution unchanged.
The second property to consider is the composition operation. Given $A\in \mathbb{R}^{q\times p}$ and $B\in \mathbb{R}^{r\times q}$:
\[B\circ A:  \mathcal{W}_p(n,\Sigma)\to  \mathcal{W}_r(n,BA\Sigma A^TB^T),\]
which corresponds to the linear map $W\mapsto BA\Sigma A^TB^T$.

\smallskip 

The category is equipped with a symmetric monoidal structure that encodes the combination of independent Wishart matrices.

\smallskip

{ \it Tensor Product ($\otimes$):}

\smallskip 

\begin{itemize}
  \item On \textbf{Objects:} For \( W_{p_1}(n_1, \Sigma_1) \) and \( W_{p_2}(n_2, \Sigma_2) \), the tensor product is defined as:
  \[
  W_{p_1}(n_1, \Sigma_1) \otimes W_{p_2}(n_2, \Sigma_2) = W_{p_1 + p_2}(n_1 + n_2, \Sigma_1 \oplus \Sigma_2),
  \]
  where \( \Sigma_1 \oplus \Sigma_2 \) denotes the block-diagonal matrix:
  \[
  \begin{pmatrix}
  \Sigma_1 & 0 \\
  0 & \Sigma_2
  \end{pmatrix}.
  \]

  \item On \textbf{ Morphisms:} For morphisms
  \[
  A_1: W_{p_1}(n_1, \Sigma_1) \to W_{q_1}(n_1, A_1 \Sigma_1 A_1^T), \quad
  A_2: W_{p_2}(n_2, \Sigma_2) \to W_{q_2}(n_2, A_2 \Sigma_2 A_2^T),
  \]
  their tensor product is:
  \[
  A_1 \otimes A_2 = A_1 \oplus A_2 :
  W_{p_1 + p_2}(n_1 + n_2, \Sigma_1 \oplus \Sigma_2)
  \to
  W_{q_1 + q_2}(n_1 + n_2, (A_1 \oplus A_2)(\Sigma_1 \oplus \Sigma_2)(A_1 \oplus A_2)^T),
  \]
  where \( A_1 \oplus A_2 \) is the block-diagonal matrix formed by \( A_1 \) and \( A_2 \).
\end{itemize}

\smallskip

\paragraph{\it Monoidal Unit:}

\smallskip 

{\it The unit object } is \( W_0(0, 0) \), representing the trivial Wishart distribution in dimension 0. It acts as the identity for the tensor product:
\[
W_p(n, \Sigma) \otimes W_0(0, 0) \cong W_p(n, \Sigma).
\]

\smallskip 
\paragraph{\it Symmetry:}

\smallskip 
The symmetric structure is given by a braiding isomorphism that permutes block matrices. For 
\( W_{p_1}(n_1, \Sigma_1) \) and \( W_{p_2}(n_2, \Sigma_2) \), there is a natural isomorphism:
\[
\beta : W_{p_1 + p_2}(n_1 + n_2, \Sigma_1 \oplus \Sigma_2) \to W_{p_2 + p_1}(n_2 + n_1, \Sigma_2 \oplus \Sigma_1),
\]
corresponding to swapping the block structure of \( \Sigma_1 \oplus \Sigma_2 \).
\end{proof}

\medskip 

  \begin{rem}
This can be given a statistical interpretation. Morphisms model linear compression (or embedding of covariance structures).
Furthermore, the monoidal structure encodes independent combinations of covariance matrices (i.e. combining datasets from different experiments).
\end{rem}
\subsection{Wishart convex cones}
Wishart matrices are naturally associated with convex cones in the space of symmetric matrices. A cone in a vector space is a subset that is closed under multiplication by positive scalars. 

We call a Wishart cone the set of all possible Wishart matrices of a given fixed dimension. This forms a convex cone in the space of symmetric matrices of a fixed size.
Furthermore, this cone is a cone of symmetric positive semi-definite matrices. It is fundamental in \textit{optimization and quantum information}.

\section{Monge-Ampère Geometry on the Cone of Positive Definite Matrices}
The cones of positive definite matrices carry important properties of Monge--Amp\`ereness. We refer to \cite{B,V} for more information concerning their importance in optimal transport and related fields.  

\subsection{Differential-Geometric Structure}
\medskip 

We have the following: 
\begin{thm}\label{T:2}
{\it The cone $\mathcal{S}_p^+$ of symmetric positive definite matrices is a Monge--Amp\`ere domain.}
\end{thm}
\begin{proof}
A proof of this statement is available in \cite{C0} and written in more details in \cite{C-LG}, where it is connected to Landau--Ginzburg models.
Relations to semi definite programming are outlined in \cite{C}.
\end{proof}
\medskip 

The cone \( \mathcal{S}^+_p \) is a Hessian manifold equipped with the following geometric structures:

\begin{itemize}
    \item \textbf{affine-invariant metric:} For \( U, V \in T_\Sigma S^+_p \), the tangent space at \( \Sigma \), define
    \[
    g_\Sigma(U, V) = \mathrm{Tr}(\Sigma^{-1} U \Sigma^{-1} V).
    \]

    \item \textbf{Monge-Ampère volume form:} The volume form associated with the Hessian structure is
    \[
    \omega = \det(\Sigma)^{-\frac{p+1}{2}} \, d\Sigma,
    \]
    which is naturally tied to the Wishart density.

    \item \textbf{Monge-Ampère equation:} The Monge-Ampère operator
    \[
    \mathrm{MA}(\varphi) = \det(\nabla^2 \varphi)
    \]
where  $\varphi$ is a potential function and $\nabla^2 \varphi$ is the Hessian matrix. This governs geodesics in the optimal transport geometry of \( \mathcal{S}^+_p \).
\end{itemize}

\subsection{Compatibility with the Monoidal Category}

\subsubsection*{(a) Monoidal Product \((\otimes)\) and Product Geometry}

The block-diagonal concatenation \[ \Sigma_1 \oplus \Sigma_2 \] corresponds to a Riemannian product structure on \( \mathcal{S}^+_{p_1} \times \mathcal{S}^+_{p_2} \). The Monge-Ampère volume form behaves multiplicatively:
\[
\omega_{\Sigma_1 \oplus \Sigma_2} = \omega_{\Sigma_1} \otimes \omega_{\Sigma_2}.
\]

\textit{Interpretation:} The tensor product in the category \( \mathcal{W} \) of Wishart distributions aligns with the geometric product structure of Monge-Ampère domains.

\subsubsection*{(b) Morphisms and Affine Equivariance}

Linear transformations
\[
A : \cW_p(n, \Sigma) \to \cW_q(n, A \Sigma A^T)
\]
are affine maps on \( \cS^+_p \), preserving the invariance of the Monge-Ampère equation under affine reparameterizations. The Fisher metric transforms covariantly:
\[
g_{A \Sigma A^T}(A U A^T, A V A^T) = g_\Sigma(U, V).
\]

\subsection{Wishart Density and Monge-Ampère Soliton}

The Wishart density is a {\it Monge-Ampère soliton} relative to \( \omega \), with the exponent \( \frac{n - p - 1}{2} \) encoding compatibility with the volume form. The normalization constant involves the multivariate gamma function:
\[
\Gamma_p\left(\frac{n}{2}\right),
\]
arising from integrating the Monge-Ampère volume over \( \cS^+_p \).

\medskip 

A \emph{soliton} in this context refers to a solution of the Monge-Ampère equation that exhibits self-similarity under geometric transformations, such as affine maps. The structure of the Wishart density reflects this soliton property in the following ways:

\begin{itemize}
    \item \textbf{Invariance:} The density’s dependence on the determinant term
    \[
    |{\bf W}|^{\frac{n - p - 1}{2}}
    \]
    ensures compatibility with the Monge-Ampère volume form \( \omega \). When integrating the density over \( \cS^+_p \), the exponent cancels out the scaling behavior of \( \omega \), preserving the self-similarity characteristic of a soliton.

    \item \textbf{Optimal Transport:} The Monge-Ampère equation governs optimal transport on the cone \( \cS^+_p \). The exponential component of the Wishart density,
    \[
    e^{-\frac{1}{2} \operatorname{Tr}(\Sigma^{-1} {\bf W})},
    \]
    looks like a Boltzmann distribution corresponding to a quadratic transport cost. This reinforces the interpretation of the Wishart density as a stationary solution---or soliton---within this geometric framework.
\end{itemize}

\subsection{Optimal Transport Interpretation}

\begin{prop}\label{P:1} 
{\it The Monge-Ampère structure on \( \cS^+_p \) equips the category \( \mathcal{W} \) with a natural Wasserstein metric. Displacement interpolation between two objects:
\[
\cW_p(n, \Sigma_1) \quad \text{and} \quad \cW_p(n, \Sigma_2)
\]
is governed by the Monge-Ampère equation.

\textbf{Entropic Regularization:} The additivity property
\[
{\bf W}_1 + {\bf W}_2 \sim \cW_p(n_1 + n_2, \Sigma)
\]
mirrors the convolution of probability measures in optimal transport.}
\end{prop}
\begin{proof}

The notation \( W_p(n, \Sigma) \) explicitly specifies that \( \Sigma \) is a \( p \times p \) matrix. For the Wasserstein metric to compare two distributions meaningfully, their covariance structures must reside in the same geometric space, i.e., \( \Sigma_1, \Sigma_2 \in \cS^+_p \), the cone of \( p \times p \) positive definite matrices.

Morphing \( \Sigma_1 \) to \( \Sigma_2 \) via optimal transport requires a common domain for the transport map \( \nabla \varphi \), which is only possible if both matrices have the same dimension.

The 2-Wasserstein distance \( \mathscr{W}_2 \) between two covariance matrices \( \Sigma_1 \) and \( \Sigma_2 \) is given by the affine-invariant geodesic formula:
\[
\mathscr{W}_2^2(\Sigma_1, \Sigma_2) = \mathrm{Tr} \left( \Sigma_1 + \Sigma_2 - 2 \left( \Sigma_1^{1/2} \Sigma_2 \Sigma_1^{1/2} \right)^{1/2} \right),
\]
which inherently assumes \( \Sigma_1, \Sigma_2 \in \cS^+_p \). Note that this formula fails when \( \Sigma_1 \) and \( \Sigma_2 \) have different dimensions.

Regarding the Monge-Ampère equation, the displacement interpolation (geodesic) between \( \Sigma_1 \) and \( \Sigma_2 \) solves the Monge-Ampère equation:
\[
\det(\nabla^2 \varphi) = f,
\]
where \( \varphi \) is a convex potential on \( \mathbb{R}^p \) and $f$ is a given function. 

\end{proof}
\begin{rem} Wishart distributions model the scatter matrices of multivariate data. 
Let us compare \( \cW_p(n, \Sigma_1) \) and \( \cW_q(m, \Sigma_2) \)  with \( p \neq q \). This corresponds to comparing datasets of differing intrinsic dimensions. However,  a direct geometric comparison is invalid without additional structure such as embedding or projection.

If \( \Sigma_1 \in \mathbb{R}^{p \times p} \) and \( \Sigma_2 \in \mathbb{R}^{q \times q} \) with \( p \neq q \), the Wasserstein distance between \( \mathscr{W}_p(n, \Sigma_1) \) and \( \mathscr{W}_q(n, \Sigma_2) \) is not natively defined. However the following cases hold:

\begin{itemize}
    \item \textbf{Embedding:} one can extend the smaller matrix by padding with identity blocks, e.g.,
    \[
    \Sigma_1' = \begin{bmatrix} \Sigma_1 & 0 \\ 0 & I_{q-p} \end{bmatrix}, \quad \text{if } p < q.
    \]

    \item \textbf{Projection:} one can project the higher-dimensional matrix to a lower-dimensional subspace by marginalizing out variables.

\end{itemize}
\end{rem}

\subsection{Enriching the Categorical Axioms}
\begin{prop}\label{P:2} 
{\it We have the following properties:}
\begin{enumerate}
    \item \textbf{Monge-Ampère Functoriality:} Morphisms
    \[
    A: \cW_p(n, \Sigma) \to \cW_q(n, A \Sigma A^T)
    \]
    are (up to scaling) volume-preserving with respect to the Monge-Ampère volume form \( \omega =\det(\Sigma)^{-\frac{p+1}{2}}d\Sigma \).
    \item[] 
    \item \textbf{Geometric Additivity:} The monoidal product \( \otimes \) corresponds to independent coupling in the Wasserstein space over \( \cS^+_p \).

    \item[] 

    \item \textbf{Convex Duality:} The category \( \mathcal{W} \) admits a dual structure via the Legendre transform on \( \cS^+_p \), linking Wishart laws to their entropy-regularized duals.
\end{enumerate}
\end{prop}

\begin{proof}
1. {\it Monge-Ampère Functoriality.} Linear morphisms \( A: W_p(n, \Sigma) \to W_q(n, A \Sigma A^T) \) are volume-preserving up to scaling with respect to the Monge-Ampère volume form:
\[
\omega = \det(\Sigma)^{-\frac{p+1}{2}} \, d\Sigma.
\]
\begin{itemize}
\item Under a linear map \( A \in \mathbb{R}^{q \times p} \), the scale matrix transforms as \( \Sigma \mapsto A \Sigma A^T \). The volume form \( \omega \) transforms according to the Jacobian determinant of this map. For \( \Sigma \in \cS^+_p \), the pushforward measure satisfies:
    \[
    A_* \omega = \det(AA^T)^{-\frac{q+1}{2}} \cdot \det(\Sigma)^{\frac{q+1}{2} - \frac{p+1}{2}} \, d(A \Sigma A^T).
    \]

\item  If \( q = p \), and \( A \in \mathrm{GL}(p) \), then  
\( \det(AA^T) = \det(A)^2 \), so:
    \[
    A_* \omega = \det(A)^{-(p+1)} \omega,
    \]
    showing that the morphism scales the volume form by \( \det(A)^{-(p+1)} \).

 \item For general $A$, the scaling reflects the relative change in dimension. 
    The map \( A \) preserves the projective class of \( \omega \), aligning with the exponent \( \frac{n - p - 1}{2} \) in the Wishart probability density function, which adjusts for scaling under transformation.

\end{itemize}
\smallskip 

 {\it 2. Geometric Additivity.} The tensor product \( \otimes \) (i.e., block-diagonal concatenation of covariance matrices) corresponds to independent coupling in Wasserstein space over \( \cS^+_p \).

\begin{itemize}
    \item Given Wishart laws \( \cW_{p_1}(n_1, \Sigma_1) \) and \( \cW_{p_2}(n_2, \Sigma_2) \), their tensor product is:
    \[
    \cW_{p_1 + p_2}(n_1 + n_2, \Sigma_1 \oplus \Sigma_2),
    \]
    where
    \[
    \Sigma_1 \oplus \Sigma_2 =
    \begin{pmatrix}
        \Sigma_1 & 0 \\
        0 & \Sigma_2
    \end{pmatrix}.
    \]

    \item \textbf{Wasserstein Product Metric:} \\
    For independent measures \( \mu_i, \nu_i \), the Wasserstein distance satisfies:
    \[
   \mathscr{W}_2^2(\mu_1 \otimes \mu_2, \nu_1 \otimes \nu_2) =    \mathscr{W}_2^2(\mu_1, \nu_1) +    \mathscr{W}_2^2(\mu_2, \nu_2).
    \]
    This reflects the statistical independence encoded in \( \Sigma_1 \oplus \Sigma_2 \).

    \item \textbf{Optimal Transport Interpretation:} \\
    The geodesic between \( \Sigma_1 \oplus \Sigma_2 \) and \( \widetilde{\Sigma}_1 \oplus \widetilde{\Sigma}_2 \) in \( \cS^+_{p_1 + p_2} \) decomposes as:
    \[
    \text{Geodesic}_{\cS^+_{p_1}}(\Sigma_1, \widetilde{\Sigma}_1) \quad \text{and} \quad \text{Geodesic}_{\cS^+_{p_2}}(\Sigma_2, \widetilde{\Sigma}_2),
    \]
    revealing the tensor product's alignment with separable transport.
\end{itemize}

\medskip
 
 {\it 3. Convex Duality.} The category \( \mathcal{W} \) of Wishart laws admits a dual structure via the Legendre transform on \( S^+_p \), linking to entropy-regularized models.

\begin{itemize}
    \item \textbf{Legendre-Fenchel Duality:} \\
    For a convex function \( \varphi : S^+_p \to \mathbb{R} \), its Legendre transform is:
    \[
    \varphi^*(M) = \sup_{\Sigma \in S^+_p} \left\{ \operatorname{Tr}(M\Sigma) - \varphi(\Sigma) \right\}.
    \]

    \item \textbf{Wishart Partition Function:} \\
    The log-partition function of the Wishart distribution:
    \[
    \log \Gamma_p\left(\frac{n}{2}\right) - \frac{n}{2} \log |\Sigma|
    \]
    is convex in \( \Sigma^{-1} \) and admits a dual interpretation as an entropy function.

    \item \textbf{Entropy Regularization:} \\
    Minimizing expected negative log-likelihood:
    \[
    \mathbb{E}[-\log f(W)] = \operatorname{Tr}(\Sigma^{-1} \mathbb{E}[W]) - \frac{n - p - 1}{2} \log |\Sigma| + \text{const},
    \]
    connects the Wishart law to entropy-regularized optimal transport, mirroring Kullback-Leibler divergence minimization.

    \item \textbf{Transport Duality:} \\
    The Wasserstein geodesic between \( \Sigma_1 \) and \( \Sigma_2 \) is dual to a displacement-convex entropy functional governed by the Monge-Ampère equation. \end{itemize}

\end{proof}

Our last statement bridges therefore:
    \begin{enumerate}
        \item \emph{Wishart geometry (covariance interpolation),}
        \item \emph{Entropy regularization (statistical efficiency),}
        \item \emph{Convex analysis (Legendre duality).}
    \end{enumerate}

\subsection*{Conclusion}

The symmetric monoidal category \( \mathcal{W} \) of Wishart distributions is enriched by the Monge-Ampère geometry of \( \cS^+_p \), revealing:
\begin{itemize}
    \item A deep interplay between algebraic operations (additivity, linear morphisms) and geometric structures (optimal transport, Hessian metrics).
    \item A statistical interpretation of the Monge-Ampère equation as governing the ``flow'' of covariance matrices under linear transformations and additive noise.
    \item A bridge between probabilistic properties (e.g., Wishart normalization) and geometric invariants (e.g., volume forms).
\end{itemize}

This unified perspective enables new tools from optimal transport and Hessian geometry to analyze Wishart matrices, with applications in information geometry, high-dimensional statistics, and random matrix theory.

\section{Applications to Quantum Computing and Quantum Information}

Quantum systems interact with their environment, leading to decoherence. Random matrix theory, particularly Wishart ensembles, provides statistical models to study how noise affects quantum states \cite{B1,B2,Z}. The spectral properties of Wishart matrices allow for an analysis of how errors propagate in quantum systems \cite{B3}.
They help in designing quantum error correction (QEC) codes by modeling perturbations in quantum states and providing insights into the geometry of stabilizer codes. 

\smallskip 

The categorical and Monge-Ampère geometric structures of Wishart matrices suggest novel frameworks for quantum error correction by bridging matrix-valued statistics, optimal transport, and quantum information theory.  In particular, in QEC, spatially or temporally correlated errors (such as in bosonic or spin systems) could be modeled using Wishart-distributed noise, where $\Sigma$ encodes error correlations. Linear maps model noise propagation through quantum circuits, preserving algebraic structure under transformations. The Wasserstein metric $\mathscr{W}_2$ between Wishart laws quantifies the minimal ``cost'' to transform one error covariance $\Sigma_1$ into another $\Sigma_2$. This could optimize error mitigation strategies by minimizing the resource cost of correcting correlated errors.

\smallskip 

Equilibrium solutions to the Monge-Ampère equation describe noise configurations that are stable under transport. These can represent noise thresholds or critical points in fault-tolerant QEC protocols.

Finally, the monoidal product $\otimes$ (direct sum $\Sigma_1\oplus \Sigma_2$) formalizes the combination of independent error processes across qubits. This aligns with concatenated codes or surface codes, where errors on disjoint regions are treated independently.

\smallskip 

Therefore, the categorical and Monge-Ampère structures of Wishart matrices offer:

\begin{itemize}
\item Noise Modeling: Representing correlated errors as covariance structures.
\item Optimal Correction: Using Wasserstein geodesics to minimize correction costs.
\item Code Design: Leveraging convex duality and tensor products for fault tolerance.
\end{itemize}

\section*{Acknowledgments}
This research is part of the project No. 2022/47/P/ST1/01177 cofunded by the National Science Centre  and the European Union's Horizon 2020 research and innovation program, under the Marie Sklodowska Curie grant agreement No. 945339 \includegraphics[width=1cm, height=0.5cm]{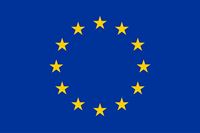}. For the
purpose of Open Access, the author has applied a CC-BY public copyright license
to any Author Accepted Manuscript (AAM) version arising from this submission.

\end{document}